\documentclass[12pt]{amsart}
\usepackage[utf8]{inputenc}
\usepackage[centering,hmargin=1in, tmargin=1.0in, bmargin=1.0in]{geometry}
\usepackage[alphabetic]{amsrefs}
\usepackage[colorinlistoftodos]{todonotes}
\usepackage{amsmath, amssymb, comment, tikz-cd, todonotes, xcolor, hyperref}

\makeatletter
\def\@settitle{\begin{center}%
  \baselineskip14\p@\relax
  \bfseries
  \uppercasenonmath\@title
  \@title
  \ifx\@subtitle\@empty\else
     \\[1ex]\uppercasenonmath\@subtitle
     \footnotesize\mdseries\@subtitle
  \fi
  \end{center}%
}
\def\subtitle#1{\gdef\@subtitle{#1}}
\def\@subtitle{}
\makeatother

\theoremstyle{plain}
\newtheorem{lemma}{Lemma}[section]
\newtheorem{theorem}[lemma]{Theorem}

\newtheorem{corollary}[lemma]{Corollary}
\newtheorem{proposition}[lemma]{Proposition}

\theoremstyle{definition}
\newtheorem{definition}[lemma]{Definition}

\theoremstyle{remark}
\newtheorem*{remark}{Remark}

\DeclareMathOperator{\rank}{\text{rank}}
\DeclareMathOperator{\Map}{\text{Map}}
\DeclareMathOperator{\sSet}{\textbf{sSet}}
\DeclareMathOperator{\sk}{\textit{sk}}

\newcommand{\Z}{\mathbb{Z}}

\newcommand{\im}{\textup{im}}
\renewcommand{\k}{\mathbf{k}}

\begin{document}

\title{Finding Bounded Simplicial Sets with Finite Homology}

\author{Preston Cranford}
\address{Massachusetts Institute of Technology, Cambridge, MA 02139, USA}
\email{prestonc@mit.edu}

\author{Peter Rowley}
\address{Massachusetts Institute of Technology, Cambridge, MA 02139, USA}
\email{rowley@mit.edu}



\maketitle

\begin{abstract}
    The central problem in computational algebraic topology is the computation of the homotopy groups of a given space, represented as a simplicial set. Algorithms have been found which achieve this, but the running times depend on the size of the input simplicial set. In order to reduce this dependence on the simplicial set chosen, we describe in this paper a procedure which, given a prime $p$ and a finite, simply-connected simplicial set with finite integral homology, finds a $p$-locally equivalent simplicial set with size upper bounded by a function of dimension and homology. Using this in conjunction with the above algorithm, the $p$-local homology can be calculated such that the running time dependence on the size of the initial simplicial set is contained in a separate preprocessing step.
\end{abstract}

\section{Introduction}

An important area of algebraic topology is the study of maps into and out of a given space $X$. In the process of understanding such maps, one of the most fundamental questions that can be asked surrounds the structure of all maps from spheres into $X$, or the study of the homotopy groups of $X$. However, despite being so fundamental, and despite their ties to homology through results such as the Hurewicz theorem, the homotopy groups of a space have proven to be extremely difficult to calculate. Because of this, the calculation of homotopy groups has become one of the central problems of computational algebraic topology.


Despite the difficulty, there do exist algorithms which compute the homotopy groups of a simply connected space which represents a finite homotopy type. When algorithmically computing homotopy groups, a topological space is represented in the computer as a simplicial set. There are various algorithms for doing this, but in large part they are all variations of an algorithm due to Brown \cite{Brown}. His algorithm calculates the homotopy groups of a simply connected space by inductively computing its Postnikov tower.

However, his algorithm is by no means practical in terms of its running time. In fact, the problem of computing $\pi_k(X)$ for a space $X$ has been shown to be W[1]-hard \cite{w1} when $k$ is regarded as an input parameter, which implies that a fast general algorithm almost certainly does not exist when varying $k$. However, by regarding $k$ as fixed, \v{C}adek et. al. \cite{alg} found an algorithm which computes $\pi_k(X)$ in time polynomial in the size of the input simplicial set $X$. Their algorithm uses Brown's as a backbone, but they associate additional data with each intermediate simplicial set, called ``polynomial time homology," which allows each of the steps to be performed in polynomial time.

Though this result is impressive, their polynomial time guarantee is slightly unsatisfactory due to its dependence on the size of the input simplicial set. In particular, this quantity has no relation to the space represented by the simplicial set, and in fact could be arbitrarily large for even the simplest underlying spaces. As such, it would be desirable to find an algorithm which, given any simplicial set, outputs another homotopy equivalent simplicial set which has size that is bounded in terms of quantities related to the space represented by the simplicial set.

Our main result achieves this goal $p$-locally for any prime $p$ with the restriction that the input space $X$ must be finite-dimensional with finite integral homology groups. The bound is given in terms of various quantities that are derived from the dimension and homology groups of the $p$-localization of $X$. To start, we use the dimension $d$ of $X$. We will also use the following quantities.

\begin{definition} Given $X$ as above and a prime $p$, we denote
\begin{equation*}
h_p:=\sum_{k=1}^{d} \rank_p{H_k(X; \Z_{(p)})}.
\end{equation*}
This quantity, the total rank of the $\Z_{(p)}$-homology of $X$, is used as a bound on the overall size of the $p$-local homology. In addition, let
\[h:=\max_p h_p.\]
Next, recall that an abelian group $A$ has $p$-torsion exponent $b$ if $b$ is the smallest integer such that $p^b$ annihilates all $p$-torsion elements of $A$. Let $m_p$ be the sum of the $p$-torsion exponents of $H_k(X)$ over the range $0\leq k\leq d$. Intuitively, this measures how much $p$-torsion is in the homology of $X$. Then, analogously to the definition of $h$, let
\[m:=\max_p m_p.\]
 Finally, let
\[N:=\prod_{p:h_p\neq0} p.\]
We use this as a measure of the diversity and size of the primes in the homology of $X$.
\end{definition}
\begin{remark}
Though all of these quantities are functions of $X$, this dependence is left implicit, as $X$ is typically clear from the context.
\end{remark}

Given these, the main result is as follows.


\begin{theorem}[Theorem \ref{thm:main}]\label{thm:intro_main}
Let $X$ be a finite, simply-connected simplicial set of dimension $d$ with finite integral homology, and let $p$ be a prime. Then there exists an algorithm giving a simplicial set which is $p$-locally homotopy equivalent to $X$ with at most
\begin{equation*}
    \exp\left(mh\log(N)\exp\left(\log(2h)d+O\left(\log(d)^3\right)\right)\right)
\end{equation*}
non-degenerate simplices.
\end{theorem}

\begin{remark}
Though $p$-local equivalence is not as strong as full homotopy equivalence, we note that if $X$ has finite homology groups, it is is still sufficient to find the homotopy groups of the space. The procedure simply involves localizing at one prime at a time and using the above result for each. Since $X$ has finite dimension and finite homology groups, this only involves localizing at finitely many primes.
\end{remark}

Our main theorem implies the following result for a simply-connected, finite homotopy type with finite integral homology.

\begin{corollary}[Corollary \ref{cor:main}]\label{cor:intro_main}
Let $X$ be a simply-connected, finite homotopy type with finite integral homology whose highest non-trivial homology is in degree $d$. Then there exists a finite simplicial set $S$ representing a space $p$-locally equivalent to $X$ with at most
\[\exp\left(mh\log(N)\exp\left(\log(2h)(d+1)+O\left(\log(d)^3\right)\right)\right).\]
non-degenerate simplices.
\end{corollary}

Our algorithm largely depends on part of Brown's algorithm; we follow the presentation given in \cite{alg}. In particular, we construct the $(d+1)^{\text{st}}$ stage of a Postnikov tower of the input space $X$, where $d$ is the dimension of $X$. With the finite homology condition as stated, this gives us a simplicial set which, after some slight modifications, is $p$-locally equivalent to $X$.


\subsection*{Paper Outline}

In section \ref{sec:prelim}, we establish notation and recall definitions. In particular, we review the definitions of Eilenberg-Maclane spaces, Postnikov systems, and twisted products, and present the standard simplicial set representative for an Eilenberg-Maclane space.

Next, in section \ref{sec:bounding_homotopy} we bound the size of the homotopy groups of a simply-connected, finite-dimensional space $X$ with finite integral homology; the bound is in terms of the size of the homology groups. This is accomplished by combining two results from the literature: one that bounds the rank of the homotopy groups and another that bounds the order of elements of the homotopy groups.

In section \ref{sec:alg} we use the Postnikov system algorithm of Brown to create a simplicial set that will be shown to satisfy the requirements of Theorem \ref{thm:intro_main}. This is accomplished by finding a skeleton of a Postnikov stage of $X$ via Brown's algorithm, then removing simplices so that the Postnikov map induces a homotopy equivalence.

In section \ref{sec:postnikov} we bound the size of the resulting simplicial set by bounding the size of a Postnikov stage of $X$. This is done by analyzing the size of the standard simplicial set model of an Eilenberg-Maclane space and then expressing sizes of Postnikov stages in terms of these spaces.

Finally, in section \ref{sec:future} we consider potential ways to strengthen our main result in the future, such as improving the bound in Theorem \ref{thm:intro_main} and extending it to include spaces with infinite homology groups. We also discuss potential directions that could be considered in pursuit of either of these goals.


\subsection*{Conventions}
Let $S$ be any simplicial set. Its set of $k$-simplices is denoted $S_k$, the face operators $S_k\rightarrow X_{S-1}$ are denoted $d_0, \dots, d_k$, and the degeneracy operators $S_k\rightarrow S_{k+1}$ are denoted $s_0, \dots, s_k$. Given any simplicial set $Y$, we denote by $\sk_k(Y)$ the $k$-skeleton of $Y$. We use $C_*(Y)$ to denote the normalized chain complex of $Y$, which is generated by the set of nondegenerate simplices in each degree. The chain differential map $C_n\rightarrow C_{n-1}$ will be denoted by $\partial_n$, or just $\partial$ if $n$ is clear from the context.

In this paper $X$ will denote a fixed finite, simply-connected simplicial set of dimension $d$ with finite integral homology unless otherwise indicated. We will denote by $d$ the dimension of $X$. Though this overlaps with the notation for face maps, it will always be clear from context which one is meant.

We take all exponent bases and logarithm bases to be $e$ unless otherwise indicated.
\subsection*{Acknowledgements}

We thank Robert Burklund for suggesting this project and mentoring us. We also thank the organizers of Summer Program in Undergraduate Research at MIT; in particular, we would like to thank David Jerison and Ankur Moitra for meeting with us weekly throughout the course of of the project. Finally, we thank Haynes Miller for his advice regarding simplicial groups.

\section{Preliminaries}\label{sec:prelim}

In this section we establish notation and recall definitions. In particular, we review the definitions of Eilenberg-Maclane spaces, Postnikov systems, and twisted products, and present the standard simplicial set representative of an Eilenberg-Maclane space.

\subsection{Eilenberg-Maclane Spaces and Postnikov Systems}

We begin by defining Eilenberg-Maclane spaces and presenting the standard simplicial models for them. We will then use these concepts to define a Postnikov system of a space.

\begin{definition}
Let $\pi$ be an Abelian group and let $n\geq 1$. An \textit{Eilenberg-Maclane Space} $K(\pi, n)$ is a space such that $\pi_n(K(\pi, n))=\pi$ and $\pi_k(K(\pi, n))=0$ for all $k>0, k\neq n$.
\end{definition}

It is a standard fact of algebraic topology that there exists a $K(\pi, n)$ and it is unique up to homotopy equivalence among spaces homotopy equivalent to CW complexes.

We now present the standard simplicial model for $K(\pi, n)$. From here on in this paper, the notation $K(\pi, n)$ will refer to this simplicial set, defined as follows.
\begin{definition}
Let $\Delta^\ell$ denote the simplicial set representing the standard $\ell$-simplex. Then we define $K(\pi,n)_\ell$ as the set of cocycles $Z^n(\Delta^\ell; \pi)$. The face operator $d_i$ is defined in the following way. Each simplex $\sigma \in K(\pi, n)$ is represented by a cocyle in $Z^n(\Delta^l, \pi)$, so let $d_i \sigma$ be the restriction of $\sigma$ to the $i^{\text{th}}$ $(l-1)$-face of $\Delta^l$. The degeneracy operators are defined similarly.
\end{definition}

This definition rise to a similar, yet slightly larger, simplicial set $E(\pi, n)$. This is defined exactly as above, except that $E(\pi, n)_\ell$ is the set of cochains $C^n(\Delta^\ell;\pi)$, rather than the cocycles. Since $E(\pi, n)$ is defined as a set of cochains and $K(\pi, n)$ is defined as a set of cocycles, there is a simplicial map $\delta_*:E(\pi,n)\rightarrow K(\pi, n+1)$ which is induced by the coboundary map $\delta: C^n(\Delta^\ell;\pi)\rightarrow Z^{n+1}(\Delta^\ell;\pi).$
We can now define a Postnikov system of a space.

\begin{definition}
Let $X$ be a simplicial set. A \textit{Postnikov system} of $X$ is a collection of simplicial sets $\{P_k\}_{k=0}^{\infty}$ and maps $\varphi_k:X\rightarrow P_k$ and $p_k:P_k\rightarrow P_{k-1}$ which fit in a commutative diagram as in Figure \ref{fig:postnikov_tower}, such that the $p_k$ are fibrations with fiber $K(\pi_k(X),k)$. We call $P_k$ the $k^{\text{th}}$ \textit{Postnikov stage}, and call $p_k$ the $k^{\text{th}}$ \textit{projection map}.

\begin{figure}[h!]
\begin{center}
\begin{tikzcd}
    & \vdots\arrow[d, "p_3"]\\
    & P_2\arrow[d, "p_2"]\\
    & P_1\arrow[d, "p_1"]\\
    X \arrow[r, "\varphi_0", swap] \arrow[ur, "\varphi_1", swap] \arrow[uur, "\varphi_2"] & P_0
\end{tikzcd}
\end{center}
\caption{Postnikov Tower}
    \label{fig:postnikov_tower}
\end{figure}
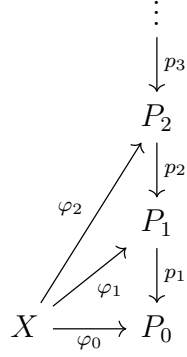

The maps $\varphi_k$ are such that the induced maps $\pi_n(\varphi_k):\pi_n(X)\rightarrow\pi_n(P_k)$ are isomorphisms for $n\leq k$. In addition, for each $k\geq 1$ there is a map $\k_{k-1}:P_{k-1}\rightarrow K(\pi_{k}(X),k+1)$ called the $(k-1)^{\text{st}}$ \textit{Postnikov class}; based on this, $P_{k}$ is obtained via the pullback in Figure \ref{fig:pullback}.

\begin{figure}[h!]
    \centering
    \begin{tikzcd}
        P_k \arrow[r] \arrow[d] & E(\pi_k(X), k)\arrow[d, "\delta_*"]\\
        P_{k-1} \arrow[r, "\textbf{k}_{k-1}", swap] & K(\pi_k(X), k+1)
    \end{tikzcd}
    \caption{Pullback square for $P_k$}
    \label{fig:pullback}
\end{figure}
\end{definition}



\subsection{Twisted Products}



A useful construction in the analysis of Postnikov systems is the twisted Cartesian product. According to \cite{alg}, this construction is intended to mimic a fiber bundle in simplicial sets; they do not give details, but see their paper for slightly more of the intuition.

\begin{definition}
Let $F, B$ be a simplicial sets, and let $F \times B$ denote their Cartesian product. Let $G$ be a simplicial group acting on $F$ on the right. A \textit{twisting operator} $\tau$ is a collection of maps $(\tau_k)_{k=1}^{\infty}$ with $\tau_k:B_k\rightarrow G_{k-1}$ such that the following conditions are satisfied:

\begin{itemize}
    \item $d_0(\tau(\beta))=\tau(d_1 \beta)\tau(d_0 \beta)^{-1}$
    \item $d_i(\tau(\beta))=\tau(d_{i+1} \beta)$ for $i\geq 1$
    \item $s_i(\tau(\beta))=\tau(s_{i+1} \beta)$ for all $i$
    \item $\tau(s_0\beta)=e_k$ for all $\beta\in B_k$, where $e_k$ is the identity of $G_k$.
\end{itemize}
\end{definition}

\begin{definition}
Let $F$, $B$, and $G$ be as above, and let $\tau$ be a twisting operator. Then the \textit{twisted Cartesian product} (or \textit{twisted product}) $F \times_{\tau} B$ is a simplicial set with $(F\times_\tau B)_k=F_k\times B_k$; its face and degeneracy maps coincide with those of the Cartesian product, with the exception of $d_0$, which satisfies

\begin{equation*}
    d_0(f,b)=(d_0(f)\tau(b), d_0(b)).
\end{equation*}
\end{definition}

The map $\delta_*: E(\pi,k)\rightarrow K(\pi, k+1)$ described above is a fibration with fiber $K(\pi, k+1)$. In keeping with the intuition that twisted products represent fiber bundles, therefore, it turns out that $E(\pi,k)$ can be described as a twisted product $K(\pi, k) \times_{\tau} K(\pi, k+1)$. The twisting operator $\tau$ is described as follows. Let $z\in Z^{k+1}(\Delta^l; \pi)$, and let a $k$-face of $\Delta^l$ be denoted by a $(k+1)$-tuple $(i_0,\dots, i_k)$ with $0\leq i_0<\dots <i_k\leq\ell$. Then we define

\begin{equation*}
    (\tau(z))(i_0,\dots, i_k):=z(0, i_0+1, i_1+1, \dots, i_k+1)-z(1,i_0+1, i_1+1, \dots, i_k+1)
\end{equation*}

It is a consequence of this fact that the pullback shown in Figure \ref{fig:pullback} gives rise to an isomorphism as follows, where $\tau^*:=\tau\circ \k_{k-1}$:
\begin{equation*}\label{eqn:twistedIso}
    P_k\cong P_{k-1}\times_{\tau^*}K(\pi_k(X), k+1).
\end{equation*}
 We refer the reader to \cite{May}*{Proposition 18.7} for more details.

\section{Bounding Homotopy Groups}\label{sec:bounding_homotopy}

In this section we bound the size of the homotopy groups $\pi_n(X)$ of a space $X$ in terms of it's homology invariants $h,n,M$. As in the main theorem, we will only consider the case that $X$ has finitely many nonzero homology groups, each of which is of finite order. It is well known by \cite{Serre} that this implies that the homotopy groups are finite as well. In order to bound $|\pi_n(X)|$ based on homology information in this case, we need two separate bounds: a bound on the rank of $\pi_n(X)$, and a bound on the orders of the individual elements.

Let $X_{(p)}$ denote the $p$-localization of $X$. Since $\pi_n(X)$ is finite for all $n$, we have $X\cong\prod_p X_{(p)}$, so $\pi_n(X)\cong\prod_p \pi_n(X_{(p)})$; as such, we will focus on bounding the homotopy groups of $X_{(p)}$ for each $p$.  Note the following bound for the rank of $\pi_n(X_{(p)})$:

\begin{proposition}[\cite{Burkland}]
For $X$ simply-connected and for each prime $p$ we have that
\begin{equation*}
\rank(\pi_n(X_{(p)}))\leq f(n)r^n,
\end{equation*}
where $f(n)\in\exp\left(O\left(\log(n)^3\right)\right) $ and $r\leq h_p$.
\end{proposition}

We also have the following bound on the torsion exponent of $\pi_n(X_{(p)})$.
\begin{proposition}[\cite{Barratt}, pg. 125]
Let $X$ be simply connected. Then for each prime $p$ we have
\begin{equation*}
    p^{cn}\pi_n(X_{(p)})=0,
\end{equation*}
where $c\leq 2m_p$.

\end{proposition}

Putting these two bounds together, we get the following:

\begin{proposition}\label{prop:homotopy_group_bound}
Let $X$ be simply connected, and let $p$ be any prime. Then
\begin{equation*}|\pi_n(X_{(p)})|\leq \exp(m_p\log(p)\exp(n\log(h_p)+O\left(\log(n)^3\right)).\end{equation*}
\end{proposition}

\begin{proof} The above bounds are combined as follows.
\begin{align*}|\pi_n(X_{(p)})|&\leq{(p^{cn})}^{\rank(\pi_n(X_{(p)}))}\\
&\leq p^{cnr^nf(n)}\\
&\leq p^{2m_pnh_p^nf(n)}\\
&\leq \exp\left(\log(p)2m_pnh_p^n\exp \left(O\left(\log(n)^3\right)\right))\right)\\
&\leq \exp\left(m_p\log(p)\exp\left(\log(2nh_p^n)+O\left(\log(n)^3\right)\right)\right)\\
&\leq \exp\left(m_p\log(p)\exp\left(\log(2n)+n\log(h_p)+O\left(\log(n)^3\right)\right)\right)\\
&\leq \exp(m_p\log(p)\exp\left(n\log(h_p)+O\left(\log(n)^3\right)\right).
\end{align*}
\end{proof}

We can now easily combine the above bound at each value of $p$ to get the following overall bound.

\begin{proposition}\label{prop:full_homotopy_bound}
Let $X$ be simply connected, and let $p$ be any prime. Then
\[|\pi_n(X)|\leq \exp\left(m\log(N)\exp(n\log(h)+O\left(\log(n)^3\right)\right)\]
\end{proposition}
\begin{proof}
As noted above, we know that $\pi_n(X)\cong\prod_p \pi_n(X_{(p)})$. Therefore, to bound $\pi_n(X)$ multiply the bounds given by Proposition \ref{prop:homotopy_group_bound} at each prime $p$ which shows up in the homology of $X$. We first replace each $m_p$ with $m$, because $m_p\leq m$ for all $p$ by definition. Translating this product into a sum in the exponent replaces the $\log(p)$ with $\sum_p \log(p)$, and then combining the logs allows us to replace this with $\log(N)$. Recall also that by definition, $h_p\leq h$ for all $p$, so we can replace each $\log h_p$ with $\log h$, giving us the above expression.
\end{proof}

\section{Simplicial Set Algorithm}\label{sec:alg}
In this section we describe the algorithm used to construct the simplicial set required for the main theorem. This procedure relies heavily on the algorithm of Brown \cite{Brown} for calculating the homotopy groups of an arbitrary simply-connected simplicial set, as presented in \cite{alg}. We will describe this algorithm in minimal detail.

\subsection{Postnikov Tower Algorithm}
We begin by summarizing at a high level the algorithm of Brown \cite{Brown} for constructing a Postnikov system for the input space $X$. As always, $X$ is a finite, simply-connected simplicial set; their algorithm does not require the finiteness assumption on the homology of $X$.

The algorithm proceeds inductively, starting with the trivial zeroth stage $P_0$, which is a point, as well as $\varphi_0$ the constant map. Given $P_{k-1}$ and $\varphi_{k-1}$, the authors simultaneously find the group $\pi_k(X)$ and a simplicial set representative of the Postnikov class $\k_{k-1}:P_{k-1}\rightarrow K(\pi_k(X),k+1)$ using an algebraic construction. The $k^{\text{th}}$ Postnikov stage is then constructed as the simplicial set pullback shown in Figure \ref{fig:pullback}.


The construction of $\varphi_k$ is not important to the rest of the algorithm, so we refer the reader to \cite{alg}.

Notice, in particular, that $P_k$ is constructed as a pullback in a way that automatically gives us the following relation:

\[P_k\subset P_{k-1}\times E(\pi_k(X), k)\]
Inductively, this implies that it is true that
\begin{equation}\label{eq:postnikovBound}
    P_k\subset E(\pi_2(X),2)\times\dots\times E(\pi_k(X),k).
\end{equation}
Note that $E(\pi_1(X), 1)$ is omitted because $X$ is simply connected. Equation \eqref{eq:postnikovBound} is the most important takeaway from this algorithm as we move forward.

\subsection{Constructing an Equivalent Simplicial Set}
We now describe how the above algorithm is used in order to find a size-bounded simplicial set which is $p$-locally homotopy equivalent to $X$, as in the main theorem. In the following section we analyze the size of the simplicial set that results.

Our procedure is motivated by the following property of $p$-localization.

\begin{theorem}\cite{Bousfield}
Let $X$ be as in the main theorem, let $Y$ be a finite simplicial set , and let $f:X\rightarrow Y$ be a map which induces isomorphisms on $\Z_{(p)}$-homology in every degree. Then $f$ induces a homotopy equivalence on $p$-localizations.
\end{theorem}


The construction of a map $f$ and a space $Y$ as above uses the Postnikov tower algorithm as follows. Let $d$ be the dimension of $X$. Starting with $X$, we run the algorithm to get $P_{d+1}$, the $(d+1)^{\text{st}}$ stage of the Postnikov tower. The algorithm also gives us a map $\varphi_{d+1}: X\rightarrow P_{d+1}$ which induces isomorphisms in $\pi_n$ for $n\leq d+1$. This gives us the following.

\begin{lemma}\label{lem:homologyIso}
The map $\varphi_{d+1}:X\rightarrow P_{d+1}$ induces an isomorphism in homology through degree $d+1$. 
\end{lemma}

\begin{proof}\label{lem:finiteness}
Since $\varphi_{d+1}$ induces homotopy isomorphisms in degrees $\leq {d+1}$, it follows by the homotopy long exact sequence that after replacing $P_{d+1}$ with the mapping cylinder of $\varphi_{d+1}$, $\pi_n(P_{d+1},X)=0$ for $n\leq {d+1}$. Therefore, the relative Hurewicz theorem tells us that $H_n(P_{d+1},X)=0$ for $n\leq {d+1}$, which implies that $\varphi_{n+1}$ induces integer homology isomorphisms in degrees $\leq {d+1}$ as well.
\end{proof}

The construction of $P_{d+1}$ also gives us a crucial finiteness condition.

\begin{lemma}
The simplicial set $P_{d+1}$ is levelwise finite.
\end{lemma}
\begin{proof}
Recall from \eqref{eq:postnikovBound} that
\[P_{d+1}\subset E(\pi_2(X),2)\times\dots\times E(\pi_{d+1}(X),d+1).\]
Since the $n$-simplices of $E(\pi_k(X),k)$ are the $k$-cochains $C_k(\Delta^n, \pi_k(X))$ and $\pi_k(X)$ is a finite group by assumption, it follows that there are finitely many $n$-simplices in $E(\pi_k(X),k)$ for each $n$ and $k$. Therefore, $P_{d+1}$ is a levelwise-finite simplicial set.
\end{proof}

Based on these results, we can get a finite simplicial set which satisfies almost all of our requirements by considering $\sk_{d+2}(P_{d+1})$, the $(d+2)$-skeleton of $P_{d+1}$. It is clear that this preserves the homology through degree $d+1$, so Lemma \ref{lem:homologyIso} still holds. We will continue to use the notation $\varphi_{d+1}$ for the map $X\rightarrow \sk_{d+2}(P_{d+1})$, because the map itself has not changed. Note, in particular, that $\varphi_{d+1}$ induces homology isomorphisms in every degree other than $d+2$.



What remains, therefore, is to remove some of the $(d+2)$-simplices of $sk_{d+2}(P_{d+1})$ to create a sub-simplicial-set $Y$ with $H_{d+1}(Y;\Z_{(p)})=H_{d+2}(Y;\Z_{(p)})=0.$ The map $X\rightarrow Y$ induced by $\varphi_{d+1}$ will therefore induce $\Z_{(p)}$-homology isomorphisms in every degree, so this is sufficient. To do this, we wish to remove enough $(d+2)$-simplices from $sk_{d+2}(P_{d+1})$ such that there are no more $(d+2)$-cycles, but so that all $(d+1)$-boundaries remain boundaries. This is done in the following lemma.

\begin{lemma}\label{lem:simplicialSubset}
We can find a subset $\{t_1,\dots,t_\ell\}$ of the $(d+2)$-simplices of $sk_{d+2}(P_{d+1})$ such that, when the $t_i$ are regarded as elements of $C_{d+2}(sk_{d+2}(P_{d+1});\Z_{(p)})$, the set $\{\partial t_1,\dots,\partial t_\ell\}$ is a basis of $\im\partial_{d+2}\subset C_{d+1}(sk_{d+2}(P_{d+1});\Z_{(p)})$.
\end{lemma}
\begin{proof}
Note that since $H_{d+1}(sk_{d+2}(P_{d+1});\Z_{(p)})=0$, it follows that the map $\partial_{d+2}$ surjects onto $Z_{d+1}(sk_{d+2}(P_{d+1});\Z_{(p)})$. Let $\textbf{B}_Z$ be a basis of of $Z_{d+1}(sk_{d+2}(P_{d+1});\Z_{(p)})$. We construct the $\Z_{(p)}$-matrix representation $M$ of the chain differential $\partial_{d+2}$, using the standard basis of $C_{d+2}$ and the basis $\textbf{B}_Z$ of the target space. It is sufficient to find a subset of the columns of $M$ that form a basis of the $\Z_{(p)}$-span of the columns.

We will find this subset by induction on the size of $\textbf{B}_Z$, which is the number of rows of $M$. For the base case of size $1$, note that in a $1\times n$ matrix representing a surjective map, at least one of the entries must not be divisible by $p$. Then, since we have $\Z_{(p)}$-coefficients, that entry is invertible, so it generates the span of all of the entries of the matrix. Thus, our desired simplex $t_1$ is the simplex corresponding to the column of this entry.

Now assume that $M$ has $\ell$ rows. Since it represents a surjective map, at least one of the columns is not divisible by $p$. It follows that, using $\Z_{(p)}$ coefficients, we can scale such a column to be not divisible by any prime.


Let $t_1$ be a $(d+2)$-simplex corresponding to a column which is not divisible by $p$. It is clear that, since $t_1$'s column can be scaled to be indivisible by any prime, $\im(\partial)/\partial t_1$ has dimension $\ell-1$. In addition, the induced map $\partial^1:C_{d+2}(sk_{d+2}(P_{d+1});\Z_{(p)})/t_1\rightarrow Z_{d+1}(sk_{d+2}(P_{d+1});\Z_{(p)})/\partial t_1$ is clearly still surjective. Thus, we can find $t_2, \dots, t_\ell$ using the inductive hypothesis, and we are done.
\end{proof}

We can now put together the above results to prove the following lemma, which provides the simplicial set required by the main theorem; the remainder of the proof will consist of analyzing the size of the Postnikov stage, which takes place in the next section.

\begin{lemma}\label{lem:finite_simplicial_set}
Let $X$ be a finite, simply-connected simplicial set with finite integral homology, and let $p$ be a prime. Then there exists an algorithm giving a finite simplicial set $Y\subset sk_{d+2}(P_{d+1})$ where $Y$ is $p$-locally homotopy equivalent to $X$.
\end{lemma}

\begin{proof}
As above, we can use Brown's algorithm \cite{Brown} to construct $sk_{d+2}(P_{d+1})$. There exists a finite subset $T\subset (P_{d+1})_{d+2}$ of non-degenerate simplices such that $\{dt|t\in T\}$ is a basis of $Z_{d+1}(P_{d+1};\Z_{(p)})$. Let $Y$ be a simplicial set such that $Y_n=(P_{d+1})_{n}$ for $n\leq d+1$ and $Y_{d+2}=T$. Then the map $\varphi_{d+1}:X\rightarrow Y$ is easily seen to induce $\Z_{(p)}$-homology isomorphisms in all degrees other than $d+1$ and $d+2$. In addition, since $B_{d+1}(Y;\Z_{(p)})=B_{d+1}(P_{d+1};\Z_{(p)})=Z_{d+1}(P_{d+1};\Z_{(p)})$, we have $H_{d+1}(Y;\Z_{(p)})=0$, so $\varphi_{d+1}$ induces an isomorphism in degree $d+1$. Finally, $T$ is linearly independent, so we have that $Z_{d+2}(Y;\Z_{(p)})=H_{d+2}(Y;\Z_{(p)})=0$, meaning $\varphi_{d+1}$ induces an isomorphism in degree $d+2$ as well, which completes the proof by Lemma \ref{lem:homologyIso}.
\end{proof}

\section{Bounding Postnikov Stage Size}\label{sec:postnikov}

In this section we bound the the cardinality of the simplicial set formed in the algorithm for the main theorem, as described in Section \ref{sec:alg}.

Since the Postnikov stages $P_k$ are described in terms of the simplicial set models of the Eilenberg-Maclane spaces $K(\pi, k)$, we first seek to understand $|K(\pi, k)_n|$ for a given finite group $\pi$.

\begin{proposition} In the standard simplicial set model of $K(\pi,k)$ we have
\begin{equation*}
    |K(\pi,k)_n|=|\pi|^{\binom{n}{k}}.
\end{equation*}
\end{proposition}

\begin{proof}

Recall that elements of $K(\pi, k)_n=Z^k(\Delta^n;\pi)$ are labellings of the $k$-faces of $\Delta^n$ with elements of $\pi$ in such a way that satisfies the cocycle condition. Representing $\Delta^n$ as the set $[n]=\{0,1,\dots,n\}$, this is equivalent to labelling all of the $(k+1)$-element subsets of this set with elements of $\pi$. For convenience we will always label subsets with the elements in increasing order: $\{a_0, a_1, \dots, a_k\}$ implicitly has $a_0<a_1<\dots<a_k.$

We claim that assigning labels to each subset of the form $\{0, a_1, \dots, a_k\}$ uniquely determines a valid $k$-cocycle on the $n$-simplex. We will begin by showing uniqueness, then check that it is a valid cocycle. Let our cocycle be called $z$, and consider the face $\{a_0, a_1, \dots, a_k\}$ with $a_0\neq0$. The cocycle condition tells us that the values of $z$ on the boundary of the $(k+1)$-face $\{0, a_0, \dots, a_k\}$ add to zero in the alternating sum of the coboundary map. Rearranging this condition, we get
\begin{equation}\label{eqn:cycleDef}
    z(\{a_0,\dots,a_k\})=\sum_{i=0}^{k}(-1)^i z(\{0, \dots, \hat{a}_i,\dots\}).
\end{equation}

The expression in the summation represents omitting only $a_i$ from $\{0, a_0, \dots, a_k\}.$ The right-hand-side of \eqref{eqn:cycleDef} is clearly a well-defined value, so we have uniqueness.

Now we just have to show that the resulting cochain is indeed a cocycle. This amounts to checking that every $(k+1)$-face not containing vertex zero is labelled with zero under the coboundary map. To this end, consider the face $\{a_0, \dots, a_{k+1}\}$. We must show that
\begin{equation}\label{eqn:consistency}
    \sum_{i=0}^{k+1}(-1)^i z(\{\dots,\hat{a}_i,\dots\})=0.
\end{equation}
We will substitute each term using equation \eqref{eqn:cycleDef}. In the resulting summation, it is clear that all terms are multiples of $z(\{0, \dots, \hat{a}_i,\dots,\hat{a}_j,\dots\})$ for $0\leq i < j\leq k+1$. Fix some $i$ and $j$, and consider the coefficient of this term. It is obtained in two ways: omitting $a_j$ in \eqref{eqn:consistency} and then omitting $a_i$ in \eqref{eqn:cycleDef}, or vice versa. For the first way, we omit the $j^{\text{th}}$ term and then the $i^{\text{th}}$ term, giving a coefficient of $(-1)^j(-1)^i=(-1)^{i+j}$. For the second way, we omit the $i^{\text{th}}$ term and then, since $j>i$, we omit the $(j-1)^{\text{st}}$ term, for a coefficient of $(-1)^i(-1)^{j-1}=(-1)^{i+j-1}$. Therefore, the overall coefficient of this term is
$$(-1)^{i+j}+(-1)^{i+j-1}=0.$$
Since this term was chosen arbitrarily, the overall expression \eqref{eqn:consistency} is zero, as desired.

We now know that, as claimed, labelling each subset of $[n]$ containing $0$ gives a unique $k$-cocycle on the $n$-simplex. Since there are $\binom{n}{k}$ total $k$-simplices of the form $\{0,a_1,\dots,a_k\}$, it follows that
\begin{equation*}
    |K(\pi,k)_n|=|\pi|^{\binom{n}{k}}.
\end{equation*}
\end{proof}

Now we are prepared to  bound the size of $\sk_n(P_k)$ using our understanding of the sizes of the simplicial set models of the $K(\pi, k)$s.

\begin{proposition}\label{prop:P_d_homotopy_groups}
Let $P_k$ be the $k^{\text{th}}$ Postnikov stage of a simplicial set $X$. Then
\begin{equation*}|(\sk_n(P_k))|\leq \sum_{\ell=0}^{n}\prod_{j=2}^{k}|\pi_j(X)|^{\binom{\ell}{j}}.\end{equation*}
\end{proposition}
\begin{proof}
We proceed by induction on $k$. Recall the pullback in Figure \ref{fig:pullback} used to construct $P_k$ from $P_{k-1}$.

As mentioned in \cite{Brown}, the fibration $\delta_*: E(\pi_k(X), k)\rightarrow K(\pi_k(X), k+1)$ is a fiber bundle with fiber $K(\pi_k(X), k)$. Therefore, since it is surjective, for each $p\in P_{k-1}$ the subset of $P_k$ corresponding to it in the pullback is isomorphic to $K(\pi_k(X), k)$. It follows that
\begin{equation*}
|(P_k)_\ell|\leq |(P_{k-1})_\ell||(K(\pi_k(X),k)_\ell)|.
\end{equation*}

Since $|K(\pi, k)_\ell|=|\pi|^{\binom{\ell}{k}}$, induction on $k$ gives us
\begin{align*}
|((P_k)_\ell|&\leq \prod_{j=2}^{k}|\pi_j(Y)|^{\binom{\ell}{j}}
\end{align*}
Adding together over all $\ell$ from 0 to $n$ gives us the desired expression.
\end{proof}

Now we can bound the size of $\sk_{d+2}(P_{d+1})$, which was used to construct $Y$ in Lemma \ref{lem:finite_simplicial_set}.

\begin{lemma}\label{lem:bound}
\begin{equation*}
    |(\sk_{d+2}(P_{d+1})|\leq \exp\left(mh\log(N)\exp\left(d\log(2h)+O\left(\log(d)^3\right)\right)\right)
\end{equation*}
\end{lemma}

\begin{proof}
By plugging the formula for a upper-bound for $\pi_j(X)$ from Proposition \ref{prop:full_homotopy_bound} into the formula for a upper bound for $|(\sk_{d+2}(P_{d+1})|$ from Proposition \ref{prop:P_d_homotopy_groups} we get that
\begin{align*}
    |(\sk_{d+2}(P_{d+1})| &\leq \sum_{\ell=0}^{d+2}\prod_{j=2}^{d+1}|\pi_j(X)|^{\binom{\ell}{j}}\\
    &\leq \sum_{\ell=0}^{d+2}\prod_{j=2}^{d+1}\left(\exp(m\log(N)\exp(j\log(h)+O\left(\log(j)^3\right))\right)^{\binom{\ell}{j}}\\
\end{align*}
We upper-bound the sum using $\sum_{\ell=0}^{d+2} a_\ell\leq (d+3)\max(a_\ell)$, noting that $\max(a_\ell)$ is achieved here at $\ell=d+2$.
\begin{align*}
    &\sum_{\ell=0}^{d+2}\prod_{j=2}^{d+1}\left(\exp(m\log(N)\exp(j\log(h)+O\left(\log(j)^3\right))\right)^{\binom{\ell}{j}}\\
    &\leq (d+3)\prod_{j=2}^{d+1}\left(\exp(m\log(N)\exp(j\log(h)+O\left(\log(j)^3\right))\right)^{\binom{d+2}{j}}\\
\end{align*}
We express the product of exponentials as a exponential of a sum.
\begin{align*}
    &(d+3)\prod_{j=2}^{d+1}\left(\exp(m\log(N)\exp(j\log(h)+O\left(\log(j)^3\right))\right)^{\binom{d+2}{j}}\\
    &\leq (d+3)\left(\exp\left(m\log(N)\sum_{j=2}^{d+1}\left(\binom{d+2}{j}\exp\left(j\log(h)+O\left(\log(j)^3\right)\right)\right)\right)\right)\\
\end{align*}
We upper-bound the sum using $\sum_{j=2}^{d+1} a_jb_jc_j\leq d\max(a_j)\max(b_j)\max(c_j)$.
\begin{align*}
    &(d+3)\left(\exp\left(m\log(N)\sum_{j=2}^{d+1}\left(\binom{d+2}{j}\exp\left(j\log(h)+O\left(\log(j)^3\right)\right)\right)\right)\right)\\
    &\leq (d+3)\left(\exp\left(m\log(N)d\left(\exp(\log(2)d)\exp\left((d+1)\log(h)+O\left(\log(d+1)^3\right)\right)\right)\right)\right)\\
\end{align*}
We rewrite $d+3$ as $\exp\log(d+3)$ on the outside, rewrite $d$ as $\exp\log(d)$ on the inside, write $\exp(a)\exp(b)$ as $\exp(a+b)$, and use $O\left(\log(d+1)^3\right)=O\left(\log(d)^3\right)$.
\begin{align*}
    &(d+3)\left(\exp\left(m\log(N)d\left(\exp(\log(2)d)\exp\left((d+1)\log(h)+O\left(\log(d+1)^3\right)\right)\right)\right)\right)\\
    &\leq \exp\left(\log(d+3)+m\log(N)\exp(\log(d))\left(\exp\left(\log(2)d+d\log(h)+\log(h)+O\left(\log(d)^3\right)\right)\right)\right)\\
\end{align*}
We use $a+b\leq ab$ for sufficiently large $a,b$ and write $\exp(a)\exp(b)$ as $\exp(a+b)$.
\begin{align*}
    &\leq\exp\left(\log(d+3)+m\log(N)\exp(\log(d))\left(\exp\left(\log(2)d+d\log(h)+\log(h)+O\left(\log(d)^3\right)\right)\right)\right)\\
    &\leq\exp\left(\log(d+3)m\log(N)\left(\exp\left(\log(d)+(\log(2h))d+\log(h)+O\left(\log(d)^3\right)\right)\right)\right)\\
\end{align*}
We rewrite $\log(d+3)$ as $\exp(\log\log(d+3))$ and use $\exp(a)\exp(b)=\exp(a+b)$.
\begin{align*}
    &\exp\left(\log(d+3)m\log(N)\left(\exp\left(\log(d)+(\log(2h))d+\log(h^2)+O\left(\log(d)^3\right)\right)\right)\right)\\
    &\leq\exp\left(m\log(N)\left(\exp\left(\log\log(d+3)+\log(d)+(\log(2h))d+\log(h^2)+O\left(\log(d)^3\right)\right)\right)\right)\\
\end{align*}
We use $\log\log(d+3)+\log(d)=O(\log(d)^3)$.
\begin{align*}
    &\exp\left(m\log(N)\left(\exp\left(\log\log(d+3)+\log(d)+(\log(2h))d+\log(h)+O\left(\log(d)^3\right)\right)\right)\right)\\
    &\leq\exp\left(m\log(N)\left(\exp\left((\log(2h))d+\log(h)+O\left(\log(d)^3\right)\right)\right)\right)\\
\end{align*}
We write $\exp(\log(h))$ as $h$.
\begin{align*}
    &\exp\left(m\log(N)\left(\exp\left((\log(2h))d+\log(h)+O\left(\log(d)^3\right)\right)\right)\right)\\
    &\leq\exp\left(mh\log(N)\exp\left(\log(2h)d+O\left(\log(d)^3\right)\right)\right).\\
\end{align*}
\end{proof}

Using this, we can easily prove our main theorem.
\begin{theorem}\label{thm:main}
Let $X$ be a finite, simply-connected simplicial set of dimension $d$ with finite integral homology, and let $p$ be a prime. Then there exists an algorithm giving a simplicial set which is $p$-locally homotopy equivalent to $X$ with at most
\begin{equation*}
    \exp\left(mh\log(N)\exp\left(\log(2h)d+O\left(\log(d)^3\right)\right)\right)
\end{equation*}
non-degenerate simplices.
\end{theorem}

\begin{proof}
From Lemma \ref{lem:finite_simplicial_set} we know there exists an algorithm giving a finite simplicial set $Y \subset sk_{d+2}(P_{d+1})$ which is $p$-locally homotopy equivalent to $X$. Since $Y \subset sk_{d+2}(P_{d+1})$, we know that the size of $Y$ is upper-bounded by that of $sk_{d+2}(P_{d+1})$, which Lemma \ref{lem:bound} bounds by the desired expression.
\end{proof}

Note that our main theorem implies the following result about homotopy types in general.

\begin{corollary}\label{cor:main}
Let $X$ be a simply-connected, finite homotopy type with finite integral homology whose highest non-trivial homology is in degree $d$. Then there exists a finite simplicial set $S$ representing a space $p$-locally equivalent to $X$ with at most
\[\exp\left(mh\log(N)\exp\left(\log(2h)(d+1)+O\left(\log(d)^3\right)\right)\right).\]
non-degenerate simplices.
\end{corollary}

\begin{proof}
Since $X$ is a finite homotopy type, it has a finite simplicial set representative to which we can apply the main theorem. However, we also wish to remove the dependence on dimension.  To this end, note that $X$ is simply-connected and its homology is finitely generated, so there exists a finite CW complex homotopy equivalent to $X$ with dimension $d+1$ \cite[Proposition 4C.1]{Hatcher}. It then follows by simplicial approximation that there exists a finite simplicial set homotopy equivalent to this CW complex with the same dimension. Applying the main theorem to this simplicial set gives the desired result.
\end{proof}

\section{Future Directions}\label{sec:future}

In this section we consider potential future directions of our work on this problem. There are two main extensions that we wish to consider. For one, we would like to reduce the size of the simplicial set generated in Theorem \ref{thm:main}. We would also like to generalize the algorithm to work when $X$ has infinite homology groups, ideally finding a simplicial set which is homotopy equivalent fully rather than just $p$-locally. The latter direction is useful because it allows for the algorithm to be used in the calculation of the homotopy groups of any simply connected, finite simplicial set.

For both of the goals mentioned above, a potential strategy to explore centers on finding smaller simplicial set representatives for the Eilenberg-Maclane spaces $K(\pi, n)$. Since our size analysis centers on a twisted product of such spaces, this has a good chance of allowing us to reduce our size bound. In addition, the main obstacle to using our current strategy in the general case is that when a homotopy group has $\Z$ summands, the standard simplicial model of $K(\pi, n)$ is no longer levelwise finite, meaning that our current strategy yields an infinite simplicial set. Thus, finding better simplicial set representatives for these spaces would be very helpful for improving the results of this paper.

One method to consider in pursuit of this is inspired by the method used by Kr\v{c}ál et. al. in \cite{polyEM} to equip $K(\Z,1)$ with polynomial time homology (defined in \cite{alg}; the precise definition is not important for us). They use structures called discrete vector fields.
\begin{definition}
Let $X$ be a simplicial set, and let $\tau$ be a nondegenerate simplex of $X$. A \textit{regular face} of $\tau$ is a simplex $\sigma$ such that $\sigma=d_i\tau$ for exactly one value of $i$.
\end{definition}
\begin{definition}
Let $X$ be a simplicial set. A \textit{discrete vector field} on $X$ is a set of ordered pairs $(\sigma,\tau)$ such that $\sigma$ is a regular face of $\tau$ and such that no face appears in more than one ordered pair. The $\sigma$ faces are called \textit{sources}, the $\tau$ faces are called \textit{targets}, and the remaining faces are called \textit{critical faces}.
\end{definition}

The reason that these structures may yield results is due to their apparent relation to Kan completion. When the Kan completion functor is applied to a simplicial set, it fills in the horns that exist in the set; discrete vector fields seem to identify a way to ``un-fill-in" horns that had been filled in. This is promising because Kan completion produces a homotopy equivalent space, so going in the reverse direction may also yield a homotopy equivalent space with a smaller simplicial set representative. The concept of discrete vector fields may have to be refined further in order to achieve this, however.

\bibliographystyle{amsalpha}
\bibliography{references}

\end{document}